\documentclass{amsart}
\usepackage{amsmath}
\usepackage{amsthm}
\usepackage{amsfonts}
\usepackage{amssymb}
\usepackage{bbm}
\usepackage{tikz}
\usepackage{pgf}
\usepackage{graphicx}
\usepackage{fancyhdr}
\usepackage{blkarray}
\usepackage{multirow}
\usepackage[colorlinks=true,citecolor=black,linkcolor=black,urlcolor=blue]{hyperref}
\usepackage{xifthen}
\usepackage{eqparbox}
\usepackage{comment}
\usepackage{mathtools}

\usepackage[normalem]{ulem}

\usepackage{ifxetex}
\ifxetex
  \usepackage{fontspec}
\else
  \usepackage[T1]{fontenc}
  \usepackage[utf8]{inputenc}
  \usepackage{lmodern}
\fi

\DeclareMathOperator{\st}{\,|\,}

\DeclareMathOperator{\N}{\mathbb{N}}

\DeclareMathOperator{\depth}{depth}
\DeclareMathOperator{\lk}{lk}

\newcommand{\squigs}[1]{\left\{#1\right\}}

\newtheorem{theorem}{Theorem}
\numberwithin{theorem}{section}

\newtheorem{proposition}[theorem]{Proposition}
\newtheorem{corollary}[theorem]{Corollary}
\newtheorem{conjecture}[theorem]{Conjecture}

\theoremstyle{definition}
\newtheorem{definition}[theorem]{Definition}
\newtheorem{example}[theorem]{Example}
\newtheorem{remark}[theorem]{Remark}
\newtheorem{question}[theorem]{Question}

\newcommand\ideal[1]{\left< #1 \right>}
\newcommand\set[1]{\{ #1 \}}
\newcommand\abs[1]{\left| #1 \right|}

\definecolor{orange}{rgb}{0.898, 0.621, 0.0}
\definecolor{skyblue}{rgb}{0.336, 0.703, 0.910}
\definecolor{bluishgreen}{rgb}{0, 0.617, 0.449}
\definecolor{yellow}{rgb}{0.937, 0.890, 0.258}
\definecolor{blue}{rgb}{0, 0.445, 0.695}
\definecolor{red}{rgb}{0.832, 0.367, 0}
\definecolor{purple}{rgb}{0.797, 0.473, 0.652}

\newcommand{\defword}[1]{\emph{#1}}

\usetikzlibrary{matrix,arrows,decorations.pathmorphing}
\usetikzlibrary{positioning}
\usetikzlibrary{patterns}
\usetikzlibrary{calc}
\usepackage{xstring}

\newcommand{\drawsimplex}[4][]{

\ifthenelse{\equal{#1}{}}{\def\c{black}}{\def\c{#1}}

\foreach \i in #4 \foreach \j in #4 \foreach \k in #4
    \draw [pattern=#2, pattern color=#3] (\i) -- (\j) -- (\k);
\foreach \i in #4 \foreach \j in #4
    \draw [line width=1.5pt, draw=\c, draw opacity=1] (\i) -- (\j);
}

\title{Partition and Cohen--Macaulay extenders}
\author{Joseph Doolittle}
\email{jdoolittle@tugraz.at}
\address{Institute of Geometry, TU Graz, Austria}

\author{Bennet Goeckner} 
\email{goeckner@uw.edu}
\address{Department of Mathematics, University of Washington}

\author{Alexander Lazar}
\email{alelaz@kth.se}
\address{Department of Mathematics, KTH Royal Institute of Technology, Sweden}

\begin{document}
\begin{abstract}
If a pure simplicial complex is partitionable, then its $h$-vector has a combinatorial interpretation in terms of any partitioning of the complex. Given a non-partitionable complex $\Delta$, we construct a complex $\Gamma \supseteq \Delta$ of the same dimension such that both $\Gamma$ and the relative complex $(\Gamma,\Delta)$ are partitionable. This allows us to rewrite the $h$-vector of any pure simplicial complex as the difference of two $h$-vectors of partitionable complexes, giving an analogous interpretation of the $h$-vector of a non-partitionable complex.

By contrast, for a given complex $\Delta$ it is not always possible to find a complex $\Gamma$ such that both $\Gamma$ and $(\Gamma,\Delta)$ are Cohen--Macaulay. We characterize when this is possible, and we show that the construction of such a $\Gamma$ in this case is remarkably straightforward. We end with a note on a similar notion for shellability and a connection to Simon's conjecture on extendable shellability for uniform matroids.
\end{abstract}

\maketitle

\section{Introduction}
The $h$-vector of a simplicial complex contains important and well-studied information about the complex and its associated Stanley--Reisner ring.
If a pure complex is \emph{partitionable}, then the entries of its $h$-vector are non-negative and have a combinatorial interpretation in terms of the partitioning of the face poset. In general, the $h$-vector can be described algebraically in terms of the Stanley--Reisner ring of $\Delta$, but the aforementioned combinatorial interpretation for the \(h\)-vector of a partitionable complex does not apply to non-partitionable complexes.

We introduce a new object of study, which we will use to extend the combinatorial interpretation for the \(h\)-vector.

\begin{definition}
Let $\Delta$ be a pure $d$-dimensional simplicial complex. A pure $d$-dimensional complex $\Gamma$ is a \defword{partition extender} for $\Delta$ if
\begin{itemize}
\item $\Delta \subseteq \Gamma$.
\item $\Gamma$ is partitionable.
\item The relative complex $(\Gamma, \Delta)$ is partitionable.
\end{itemize}
\end{definition}

\begin{theorem}[Theorem~\ref{main}]
Every pure simplicial complex has a partition extender.
\end{theorem}

\pagebreak
For any relative complex \((\Gamma, \Delta)\) with $\dim \Gamma = \dim \Delta$ we can write 
$$
h(\Delta) = h(\Gamma) - h(\Gamma,\Delta).
$$
When \(\Gamma\) is a partition extender for \(\Delta\), then both of the right-hand \(h\)-vectors have combinatorial interpretations. This allows us to view the $h$-vector of $\Delta$ as an ``error term'' between the $h$-vector of $\Gamma$ and the $h$-vector of $(\Delta,\Gamma)$. Specifically, every $h$-vector of a simplicial complex is the difference between the $h$-vector of a partitionable complex and the \(h\)-vector of a partitionable relative complex.

Our construction of a partition extender can be generalized to nonpure complexes. In the nonpure case, partitionability is a more subtle condition than in the pure case (see \cite{Hachimori_sequential}). However, we show that our construction satisfies strong enough properties to yield a combinatorial interpretation of the $h$-triangle of an arbitrary nonpure complex. 

We further show that if $\depth \Bbbk[\Delta] = \dim \Bbbk[\Delta] - 1$, then for any Cohen--Macaulay complex \(\Gamma\) of the same dimension that contains $\Delta$, the relative complex $(\Gamma,\Delta)$ is Cohen--Macaulay. This similarly allows us to write the $h$-vector of any such complex as the difference between the $h$-vector of a Cohen--Macaulay  complex and the \(h\)-vector of a relative Cohen--Macaulay complex. We also show that such a $\Gamma$ does not exist if the depth of $\Bbbk[\Delta]$ is any lower.

While an equivalent notion for shellability is straightforward to define, it is unclear when shellable extenders exist. They certainly cannot exist whenever $\depth \Bbbk[\Delta] < \dim \Bbbk[\Delta] - 1$, since relative shellability implies relative Cohen--Macaulayness. We conclude with a connection to Simon's conjecture on shellability of uniform matroids \cite[Conjecture 4.2.1]{Simon_Cleanness}.

In Section~\ref{sec_background}, we review standard definitions and background material. In Section~\ref{sec_constructions}, we give explicit constructions which have the required properties to make our proofs work. In Section~\ref{sec_partitions}, we provide our main result on partition extenders for pure complexes. In Section~\ref{sec_nonpure_partition}, we consider the case of nonpure partitionability. In Section~\ref{sec_CM}, we prove parallel results with the Cohen--Macaulay property in place of partitionable. In Section~\ref{sec_shell}, we survey the current state of the problem with the shellable property. In Section~\ref{sec_closing}, we discuss possible future directions of investigation.

\section{Preliminaries}\label{sec_background}
A \defword{simplicial complex} $\Delta$ is a collection of sets such that if $\sigma \in \Delta$ and $\tau\subseteq \sigma$, then $\tau \in \Delta$. The elements of $\Delta$ are called \defword{faces} of $\Delta$, and maximal faces are called \defword{facets}. If $\sigma$ is a face of $\Delta$, the \defword{dimension} of $\sigma$ is $\text{dim}(\sigma) := |\sigma| - 1$. The \defword{dimension} of $\Delta$ is defined to be the maximum of the dimensions of the faces of $\Delta$. We say that $\Delta$ is \defword{pure} if all its facets have the same dimension.
Let $\Delta$ be a $d$-dimensional simplicial complex. The \defword{f-vector} of $\Delta$ is the vector
$$f(\Delta) = (f_{-1}(\Delta),f_0(\Delta),f_1(\Delta),\dots,f_{d}(\Delta)),$$ 
where $f_i(\Delta)$ is the number of $i$-dimensional faces of $\Delta$. Note that $f_{-1}(\Delta) = 1$ unless $\Delta$ is the \defword{empty complex} $\Delta = \varnothing$. 

The \defword{h-vector} of $\Delta$ is the vector $h(\Delta) = (h_0(\Delta),h_1(\Delta),\dots, h_{d+1}(\Delta))$ , whose entries are defined by the relation
\begin{equation}\label{def:h}
\sum_{i=0}^{d+1}f_{i-1}(\Delta)(x-1)^{d-i+1} = \sum_{i=0}^{d+1}h_{i}(\Delta)x^{d-i+1}.
\end{equation}
The \defword{face poset} $P(\Delta)$ of a simplicial complex $\Delta$ is the set of all faces of $\Delta$, partially-ordered by inclusion. An \defword{interval} $I$ in a poset $P$, denoted $I = [\sigma,\tau]$, is the set of elements $e$ of $P$ such that $\sigma \leq e \leq \tau$. When this interval $I$ is itself a Boolean poset (i.e., $I \cong 2^{[k]}$ for some $k \in \mathbb{Z}_{\ge 0}$), we say it is a \defword{Boolean interval}.

Let $\Gamma$ be a simplicial complex and $\Delta$ be a subcomplex of $\Gamma$. The \defword{relative complex} $(\Gamma,\Delta)$ consists of the faces of $\Gamma$ not contained in $\Delta$. A relative complex is \emph{pure} if all its maximal faces have the same dimension. If $(\Gamma,\Delta)$ is a relative complex, we can define $f(\Gamma,\Delta) = (f_{-1}(\Gamma,\Delta),\dots,f_d(\Gamma,\Delta))$ by $f_j(\Gamma,\Delta) = f_j(\Gamma) - f_j(\Delta)$ for all $j$. We can further define $h(\Gamma,\Delta)$ via \eqref{def:h} above.

A poset $P$ is said to be \defword{partitionable} if $P$ can be written as a disjoint union of intervals $I_1 \sqcup \cdots \sqcup I_k$ such that each $I_j$ is a Boolean interval and the maximum element of each $I_j$ is a maximal element of $P$. A (relative) complex is said to be \defword{partitionable} if its face poset is partitionable.

\begin{proposition}\label{prop:h_vec} \textrm{\cite[Page 118]{Green_Stanley}}
If a pure relative complex is partitionable, then $h_i(\Gamma,\Delta)$ is the number of Boolean intervals in any partitioning of the face poset of $(\Gamma,\Delta)$ whose minimal element is an $(i-1)$-dimensional face of $(\Gamma,\Delta)$.
\end{proposition}

We note that for any simplicial complex $\Gamma$ that $(\Gamma,\varnothing) = \Gamma$, so Proposition~\ref{prop:h_vec} holds for simplicial complexes as well. There is no previously known combinatorial interpretation of the $h$-vectors for non-partitionable complexes.

The notation $[n]$ indicates the set of integers $\{1,2,\ldots,n\}$. We take as a convention that $[0] = \varnothing$. Throughout the rest of this paper, we assume that all simplicial complexes are collections of subsets of $[n]$.

If $\sigma$ is a face of $\Delta$, the \defword{link} of $\sigma$ in $\Delta$ is the simplicial complex
$$
\lk_{\Delta}(\sigma) = \{\tau \in \Delta \st \sigma \cup \tau \in \Delta, \; \sigma \cap \tau = \varnothing\}.
$$
A simplicial complex $\Delta$ is said to be \defword{Cohen--Macaulay} (over $\Bbbk$) if, for all faces $\sigma \in \Delta$,
$$\tilde{H}_i(\lk_{\Delta}(\sigma),\Bbbk) = \begin{cases}\Bbbk^{\beta_\sigma}, & i = \dim(\Delta) - \dim(\sigma) - 1 \\ 0, & \text{otherwise}\end{cases}$$
where {$\tilde{H}_i(X,\Bbbk)$ is the  $i^{th}$ reduced homology group of $X$ with coefficients in $\Bbbk$ and} $\beta_\sigma \in \N$ is the top Betti number of the link. 
By a result of Reisner \cite{Reisner}, this definition is equivalent to $\Bbbk[\Delta]$ being Cohen--Macaulay, i.e., that $\depth \Bbbk[\Delta] = \dim \Bbbk[\Delta].$  Here $\Bbbk[\Delta]$ is the \defword{Stanley--Reisner ring} (or \defword{face ring}) of $\Delta$. For a complex $\Delta$ on $n$ vertices $\Bbbk[\Delta] \coloneqq \Bbbk[x_1,\dots,x_n] / I_\Delta$ where  $I_\Delta$ is the monomial ideal generated by non-faces of $\Delta$.

Given a face $\sigma \in \Delta$, we  distinguish between the face $\sigma$ and the complex $\langle \sigma \rangle$ whose only facet is $\sigma$. If $\dim \sigma = d$, we call this latter object a \defword{$d$-simplex}.
\section{Intermediate Constructions}\label{sec_constructions}

Our main goal is to write the $h$-vector of any pure complex as the difference of $h$-vectors of two partitionable (relative) complexes. We will prove that this is always possible in Section~\ref{Section::Main}. In this section we introduce two intermediate constructions.

\begin{definition}\label{int_def}
A $(d,k)$\emph{-partition extender} is a pure $d$-dimensional simplicial complex $\Delta$ with a specified facet $F$ and a \(k\)-dimensional face $\sigma$ in $F$ such that both $(\Delta,\ideal{F})$ and $(\Delta,\ideal{F})\cup \squigs{\sigma}$ are partitionable. 
\end{definition}

\begin{remark}
It is not true that the object $(\Delta,\ideal{F})\cup \squigs{\sigma}$ in Definition~\ref{int_def} is a relative complex in general, but we can still determine whether its face poset is partitionable or not.
\end{remark}

\begin{example}
An example of a \((1,0)\)-partition extender is \(\Delta = \langle 12,23,34,24 \rangle\) with \(F=12\) and \(\sigma=2\). The face poset of \((\Delta,\ideal{F})\) is pictured below. A partitioning of this poset is given by the intervals $[23,23]$, $[3,34]$, and $[4,24]$.

\begin{center}
\begin{tikzpicture}
\node[circle,draw=black, fill=white, inner sep=1pt,minimum size=5pt] at (-1,1) (23){$23$};
\node[circle,draw=black, fill=white, inner sep=1pt,minimum size=5pt] at (0,1) (34){$34$};
\node[circle,draw=black, fill=white, inner sep=1pt,minimum size=5pt] at (1,1) (24){$24$};
\node[circle,draw=black, fill=white, inner sep=1pt,minimum size=5pt] at (-1,0) (3){$3$};
\node[circle,draw=black, fill=white, inner sep=1pt,minimum size=5pt] at (1,0) (4){$4$};
\draw (23) -- (3) -- (34) -- (4) -- (24);
\end{tikzpicture}
\end{center}

The poset of $(\Delta,\ideal{F}) \cup \squigs{\sigma}$, which has a partitioning into the intervals \([2,23]\), \([3,34]\), and \([4,24]\), is shown below.

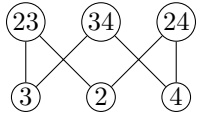
\begin{center}
\begin{tikzpicture}
\node[circle,draw=black, fill=white, inner sep=1pt,minimum size=5pt] at (-1,1) (23){$23$};
\node[circle,draw=black, fill=white, inner sep=1pt,minimum size=5pt] at (0,1) (34){$34$};
\node[circle,draw=black, fill=white, inner sep=1pt,minimum size=5pt] at (1,1) (24){$24$};
\node[circle,draw=black, fill=white, inner sep=1pt,minimum size=5pt] at (0,0) (2){$2$};
\node[circle,draw=black, fill=white, inner sep=1pt,minimum size=5pt] at (-1,0) (3){$3$};
\node[circle,draw=black, fill=white, inner sep=1pt,minimum size=5pt] at (1,0) (4){$4$};
\draw (23) -- (3) -- (34) -- (4) -- (24) -- (2) -- (23);
\end{tikzpicture}
\end{center}
\end{example}

\begin{definition}
A $(d,k)$\emph{-prepartition extender} is a pure $d$-dimensional simplicial complex $\Delta$ with a specified facet $F$, and a face $\sigma$ in $F$ of dimension $k$  such that $(\Delta,\ideal{F}) \cup \squigs{\sigma}$ is partitionable.
\end{definition}

This differs from a \((d,k)\)-partition extender in that we do not require \((\Delta,\ideal{F})\) to be partitionable.

Note that $\sigma$ is in $F$, so there are no elements below it in $(\Delta,\ideal{F}) \cup \squigs{\sigma}$. Therefore in any partitioning of the poset $(\Delta,\ideal{F}) \cup \squigs{\sigma}$, $\sigma$ must be a bottom element of some interval in the partitioning.

\begin{proposition}\label{prop:PrePartExt}
For all \(-1 \leq k \leq d\), there exists a $(d,k)$-prepartition extender.
\end{proposition}

\begin{proof}
We prove this proposition by directly constructing a \((d,k)\)-prepartition extender for arbitrary \(k\) and \(d\).
Consider two $d$-simplices, $D_1$ and $D_2$ such that \(D_1 \cap D_2 = \sigma\), where \(\sigma\) is a \(k\)-face.
Label the vertices of $D_1$ not in $\sigma$ as $\{1,\ldots,d-k\}$, the vertices of $D_2$ not in $\sigma$ as $\{d-k+1,\ldots,2d-2k\}$, and the vertices of $\sigma$ as $\{2d-2k+1,\ldots,2d-k+1\}$.

Define $W_{1,j} = \{j+1, \ldots, j+d-k+1\}$ for all \(j\) such that $0 \leq j \leq d-k-1$, and $W_{2,i}= \sigma \setminus i$ for all $i$ in $\sigma$. Let \(\Delta\) be the simplicial complex on $2d-k+1$ vertices whose facets are $D_1$, $D_2$, and all sets of the form $W_{1,j} \cup W_{2,i}$. We emphasize that when \(d=k\), there are no valid choices for \(j\), and so \(\Delta\) is the complex on \(d+1\) vertices whose facets are \(D_1\) and \(D_2\), which are in fact the same facet. We also emphasize that when \(k=-1\), there are no valid choices for \(i\), so \(\Delta\) is the complex on \(2d+2\) vertices whose facets are \(D_1\) and \(D_2\). For all other choices of \(d\) and \(k\), we see that \(|W_{1,j}| =d-k+1 \) and \(|W_{2,i}|=k \). Therefore, in all cases \(\Delta\) is a pure simplicial complex of dimension \(d\). 

The following is a set of Boolean intervals in the face poset of \(\Delta\).
\begin{align*}
I=&[\sigma,D_1]\\
I'=&[\varnothing,D_2]\\
I_{i,j}=&[\{j+1\}\cup \{v \in \sigma : v<i\},W_{1,j}\cup W_{2,i}] \text{ for } i \in \sigma \text{ and } 0 \leq j \leq d-k-1.
\end{align*}

We claim that every face of \(\Delta\) is in exactly one of these intervals, except for the face $\sigma$ which is in both $I$ and $I'$.

Note that $I \cap I' = \sigma$. Furthermore, $I$ is disjoint from each $I_{i,j}$, since every face in $I$ contains $\sigma$, and no face of $I_{i,j}$ contains $\sigma$. Likewise, $I'$ is disjoint from each $I_{i,j}$, since $j+1$ is a vertex of $D_1$ that is not contained in $\sigma$, and therefore not contained in $D_2$.

Consider some face $\tau$ not in $I$ or $I'$, that is, $\tau$ is not contained in $D_2$ and $\tau$ does not contain $\sigma$. Let $j+1$ be the least vertex of $\tau$. Since $\tau$ is not in $D_2$, this means that $j+1$ is in $[d-k]$, and so $0 \leq j \leq d-k-1$. Since \(\tau \subseteq W_{1,j'} \cup W_{2,i'}\) for some \(j',i'\), the difference between the largest index in \(\tau\) not in \(\sigma\) and \(j'+1\) is at most \(d-k\). Therefore, \(\tau \cap [2d-2k] \subseteq W_{1,j}\). Let $i$ be the largest vertex of $\sigma$ such that all smaller labeled vertices of $\sigma$ are in $\tau$. This implies that $i$ is the smallest vertex of \(\sigma\) not in $\tau$. Since $\tau \nsupseteq \sigma$, there is some vertex of $\sigma$ not in $\tau$, and therefore this $i$ exists. This shows that \(\tau \cap \sigma \subseteq W_{2,i}\). We conclude that $\tau$ is in the interval $I_{i,j}$.

Furthermore, we will show that $\tau$ is not in any other interval. By assumption, $\tau$ is not in $I$ or $I'$.

Let \(I_{i',j'}\) be an interval which contains \(\tau\). Since \(\tau\) contains all vertices of \(\sigma\) less than \(i\), and \(W_{2,i'}\) does not contain \(i'\), then \(i'\) cannot be less than \(i\), since that would imply that \(\tau\) both does and does not contain \(i'\). Likewise, $i'$ cannot be greater than $i$, since every face in $I_{i',j'}$ contains the vertices of $\sigma$ less than \(i'\), and $\tau$ does not contain $i$, which is one of those vertices. Therefore $i'=i$.

Furthermore, we see that $j'$ cannot be greater than $j$, since otherwise $W_{1,j'}$ does not contain $j+1$, and $\tau$ does contain $j+1$. Similarly, $j'$ cannot be less than $j$, because every face in $I_{i,j'}$ contains $j'+1$, but $j+1$ was the smallest vertex that $\tau$ contained. Therefore $j'=j$.

Therefore the only interval that contains $\tau$ is $I_{i,j}$.

This means that \(\Delta\) is a $(d,k)$-prepartition extender, with $D_2$ as the specified facet, $\sigma$ as the specified face, and the set $\{I\} \cup \bigcup_{i,j} \{I_{i,j}\}$ as a partition of $(\Delta,\ideal{D_2}) \cup \squigs{\sigma}$.
\end{proof}

\begin{example} We describe the facets of \((d,k)\)-prepartition extenders given in Proposition~\ref{prop:PrePartExt} for $d-2 \le k \le d$.

A $(d,d)$-prepartition extender is a $d$-simplex. 

A $(d,d-1)$-prepartition extender has the following set of facets:
\begin{align*}
D_1 &= \{1,3,4,\dots,d+2\}\\
D_2 &= \{2,3,\dots,d+2\}\\
W_{1,0}\cup W_{2,i} &= \{1,2,\dots,\hat{i},\dots,d+2\},\;\; 3\leq i \leq d+2,
\end{align*} 
where $\{1,2,\dots,\hat{i},\dots,d+2\}$ is the set $\{1,2,\dots,d+2\}\setminus \{i\}$. We therefore see that a $(d,d-1)$-prepartition extender is the boundary of the $(d+1)$-simplex on vertex set $[d+2]$.

A $(d,d-2)$-prepartition extender has the following set of facets:
\begin{align*}
D_1 &= \{1,2,5,6\dots,d+3\}\\
D_2 &= \{3,4,5,\dots,d+3\}\\
W_{1,0} \cup W_{2,i} &= \{1,2,3,5,\dots,\hat{i},\dots,d+3\}\;\; 5\leq i \leq d+3\\
W_{1,1} \cup W_{2,i} &= \{2,3,4,5,\dots,\hat{i},\dots,d+3\}\;\; 5\leq i \leq d+3.
\end{align*}

\end{example}

\begin{remark} Let $\Delta$ be a $(d,k)$-prepartition extender given in Proposition~\ref{prop:PrePartExt} with specified facet $F$ and specified $k$-face $\sigma \in F$. Then, if we define $h_{\ell}((\Delta,\langle F \rangle)\cup\{\sigma\})$ to be the number of Boolean intervals in the partitioning of $(\Delta,\langle F\rangle)\cup\{\sigma\}$ whose bottom element has size $\ell$,
$$h_{\ell}((\Delta,\langle F \rangle)\cup\{\sigma\}) = \begin{cases}d-k, & \ell < k+1\\ d-k+1, & \ell = k+1\\ 0, & \text{otherwise}\end{cases}.$$\end{remark}

\begin{proof}For all $\ell < k+1$, there are exactly $(d-k)$ intervals $I_{i,j}$ in the partitioning above whose bottom elements have size $\ell$. If $\ell = k+1$, there are $d-k$ intervals $I_{i,j}$ whose bottom elements have size $\ell$, and the interval $I = [\sigma,D_1]$ also has a bottom element whose size is $\ell$.\end{proof}

\begin{proposition}\label{dk_part}
For all \(-1 \leq k \leq d\), there exists a $(d,k)$-partition extender.
\end{proposition}

\begin{proof}
Recall from Definition~\ref{int_def} that a $(d,k)$-partition extender consists of a pure $d$-dimensional complex $\Delta$, along with a specified facet $F$ and specified $k$-dimensional face $\sigma$ in $F$. We construct our $(d,k)$-partition extender inductively, starting with \(k=d\) and decreasing \(k\). First we note that a $(d,d)$-prepartition extender is in fact a $(d,d)$-partition extender. Indeed, in any partitioning of a $(d,d)$-prepartition extender, one of the intervals must be $[\sigma,\sigma]$, and so removing $\sigma$ and this interval gives the required partitioning of $(\Delta,\ideal{F})$.

Suppose that $(d,h)$-partition extenders exist for all $h>k$. We will construct a $(d,k)$-partition extender $K$ with specified facet $F$, and specified $k$-face $\sigma$. Let $K'$ be a $(d,k)$-prepartition extender with specified facet $F$ and specified $k$-face $\sigma$.

First, fix a partitioning of \( (K', \ideal{F}) \cup \squigs{\sigma}\). Let \(\tilde{F}\) be the top element in the interval containing \(\sigma\) in this partitioning. Let $\tau$ be an $h$-face of \(K'\) such that $ \sigma \subsetneq \tau \subseteq \tilde{F}$.  By induction, there exists a $(d,h)$-partition extender $K_{\tau}$ with specified facet $F_{\tau}$ and specified $h$-face $\sigma_{\tau}$. Attach this $(d,h)$-partition extender to $K'$ by identifying $F_{\tau}$ with \(\tilde{F}\) and identifying $\sigma_{\tau}$ with $\tau$. We define $K$ to be the complex obtained from $K'$ by attaching $K_{\tau}$ for each $\tau$ with $ \sigma \subsetneq \tau \subseteq \tilde{F}$.

The complex $K$ with specified facet $F$ and specified $k$-face $\sigma$ is a $(d,k)$-partition extender. To verify this, we need a partitioning of $(K,\ideal{F}) \cup \squigs{\sigma}$ and a partitioning of $(K,\ideal{F})$. We note that \(K\) consists of a \((d,k)\)-prepartition extender \(K'\), and many \((d,h)\)-partition extenders \(K_\tau\), for each \(k<h \leq d\).

First, $(K,\ideal{F}) \cup \squigs{\sigma}$ admits a partitioning consisting of
\begin{enumerate}
\item the partitioning of $(K',\ideal{F}) \cup \squigs{\sigma}$ arising from its status as a prepartition extender,
\item the partitionings of the $K_{\tau}$ such that \(\tau\) is not included in the partitioned set.
\end{enumerate}

Furthermore, $(K,\ideal{F})$ admits a partitioning consisting of
\begin{enumerate}
\item the partitioning of $(K',\ideal{F}) \cup \squigs{\sigma}$ excluding the interval \([\sigma, \tilde{F}]\),
\item the partitionings of the $K_{\tau}$ such that \(\tau\) is included in the partitioned set.
\end{enumerate}

Since both of these partitionings exist, \(K\) is a \((d,k)\)-partition extender, and by induction, \((d,k)\)-partition extenders exist for all pairs \((d,k)\) with \(d \geq k\).
\end{proof}

Previously, we had described \((d,k)\)-prepartition extenders. Both $(d,d)$- and $(d,d-1)$-prepartition extenders are in fact partition extenders. To illustrate the full construction of a \((d,k)\)-partition extender, we give a small example in which the partition extender differs from the prepartition extender.

\begin{example}\label{PartitioningExample}
We give an example of a $(3,1)$-partition extender $K$ using the construction in Proposition~\ref{dk_part}. We start with a $(3,1)$-prepartition extender: Following Proposition~\ref{prop:PrePartExt}, we construct the prepartition extender \[K' = \ideal{1256,3456,1236,1235,2346,2345}\] with specified facet $F=3456$ and specified face $\sigma = 56.$ This labeling is identical to the canonical \((3,1)-\)prepartition extender as constructed in Proposition \ref{prop:PrePartExt}.

We observe that $K'$ is exactly the canonical $(3,1)$-prepartition extender that we constructed earlier. The following is a partitioning of $(K',\ideal{3456}) \cup \set{56}$ given by our construction:
\begin{equation}\label{eq:(3,1)PrePartExt}
    [56,1256] \hspace{1em} [1,1236] \hspace{1em} [2,2346] \hspace{1em} [15,1235] \hspace{1em} [25,2345].
\end{equation}
We now must create partition extenders for each $56 \subsetneq \tau \subsetneq 1256$, i.e., we create $(3,2)$-partition extenders for the faces $156$ and $256$. Recall that the other intervals in \eqref{eq:(3,1)PrePartExt} are fixed and will be part of both partitionings.

For the face $156$, we construct the partition extender \[K'' = \ideal{7156,2156,7256,7216,7215}\] with specified facet $F=2156$ and specified face $\sigma=156.$ The bijection to our canonical \((3,2)-\)prepartion extender is induced by $(7,2,1,5,6) \mapsto (1,2,3,4,5)$. The following is a partitioning of $(K'',\ideal{2156}).$
\begin{equation}\label{eq:(3,2)PrePartExt1}
    [7156,7156] \hspace{1em} [7,7256] \hspace{1em} [71,216] \hspace{1em} [715,7215]
\end{equation}
For the face $256$, we construct the partition extender \[K''' = \ideal{8256,1256,8156,8126,8125}\] with specified facet $F=1256$ and specified face $\sigma=256.$ The bijection to our canonical \((3,2)-\)prepartion extender is induced by $(8,1,2,5,6) \mapsto (1,2,3,4,5)$. The following is a partitioning of $(K''',\ideal{1256}):$
\begin{equation}\label{eq:(3,2)PrePartExt2}
    [8256,8256] \hspace{1em} [8,8156] \hspace{1em} [82,8126] \hspace{1em} [825,8125].
\end{equation}
The $(3,1)$-partition extender is $K = K' \cup K'' \cup K'''$ with specified facet $F = 3456$ and specified face $\sigma=56$. Equations \eqref{eq:(3,1)PrePartExt}, \eqref{eq:(3,2)PrePartExt1}, and \eqref{eq:(3,2)PrePartExt2} together give a partitioning of $(K,\ideal{3456}) \cup \set{56}$.

For a partitioning of $(K,\ideal{3456})$, we take the partitionings from equations \eqref{eq:(3,1)PrePartExt}, \eqref{eq:(3,2)PrePartExt1}, and \eqref{eq:(3,2)PrePartExt2} and modify only the first interval in each line. We get the following:

\begin{center}
\begin{tabular}{lllll}
[1256,1256] &
[1,1236] &
[2,2346] &
[15,1235] &
[25,2345] \\

[156,1567] &
[7,2567] &
[17,1267] &
[157,1257] 
\rule{0pt}{12pt}\\

[256,2568] &
[8,1568] &
[28,1268] &
[258,1258].
\rule{0pt}{12pt}
\end{tabular}
\end{center}
Thus $K$ is a $(3,1)$-partition extender.
\end{example}

\section{Main Theorem}\label{Section::Main}\label{sec_partitions}

Now we are prepared to prove our main result.

\begin{theorem}\label{main}
Every pure simplicial complex has a partition extender.
\end{theorem}

\begin{proof}
Let $\Delta$ be a pure $d$-dimensional complex. For each \(k\)-face \(\sigma\) of \(\Delta\), attach a \((d,k)\)-partition extender to \(\Delta\) by identifying \(\sigma\) and a facet containing \(\sigma\) to the specified faces of the \((d,k)\)-partition extender. Call this complex \(\Gamma\). By Proposition~\ref{dk_part}, \(\Gamma\) is a pure partitionable $d$-dimensional complex, with the partition where each \((d,k)\)-extender uses the \(\sigma\) it was attached to. Furthermore, \((\Gamma,\Delta)\) is partitionable, with the partition where each \((d,k)\)-extender is partitioned without the \(\sigma\) it was attached to. Therefore \(\Gamma\) is a partition extender for \(\Delta\).
\end{proof}

We now provide a combinatorial interpretation of the \(h\)-vector of a pure simplicial complex \(\Delta\) with a partition extender \(\Gamma\).
We can write the $f$-vector of $\Delta$ as
$$f_i(\Delta) = f_i(\Gamma) - f_i(\Gamma,\Delta).$$ 

Since the \(h\)-vector is a bijective linear transformation of the \(f\)-vector, we transform the above equation into  $$h_i(\Delta) = h_i(\Gamma) - h_i(\Gamma,\Delta).$$ Since both \(\Gamma\) and \((\Gamma,\Delta)\) are pure and partitionable, we may use the combinatorial interpretation of these values to give a combinatorial interpretation of \(h_i(\Delta)\). 

\begin{corollary}
If \(\Delta\) is a pure simplicial complex, then 
\begin{align*}
h_i(\Delta) = &|\{\text{intervals in a partitioning of \(\Gamma\) with bottom element of size}~ i\}| \\
&- |\{\text{intervals in a partitioning of \((\Gamma,\Delta)\) with bottom element of size}~ i\}|
\end{align*}
for any partition extender \(\Gamma\) of \(\Delta\).
\end{corollary}

In our construction of the partition extender $\Gamma$ of $\Delta$, there is significant overlap between the sets of intervals in the partitioning of $\Gamma$ and the partitioning of $(\Gamma,\Delta)$. Keeping track of the heights of the intervals that differ between the partitioning of $(\Gamma,\Delta)$ and that of $\Gamma$ yields
\[h_i(\Gamma) - h_i(\Gamma,\Delta) = \sum_{j=0}^{i}(-1)^{i-j}\binom{d-j}{i-j}f_{j-1}(\Delta),\]
which is exactly the formula for $h_i(\Delta)$ in terms of the $f_j(\Delta)$. Thus our construction gives a combinatorial witness to the algebraic transformation between $h(\Delta)$ and $f(\Delta)$.

\section{Nonpure Partitionability}\label{sec_nonpure_partition}
Our construction of a partition extender can be applied to nonpure complexes in a natural way. Suppose that $\Delta$ is a nonpure complex. 
If $\sigma$ is a face of $\Delta$, we write 
$$d_{\Delta}(\sigma) \coloneqq \max_{\tau \in \Delta} \{ \dim (\tau) \mid \sigma \subseteq \tau \}.$$
In \cite[Definition 3.1]{Shellable_Non_Pure_1}, Bj\"orner and Wachs define a two-dimensional array called the \emph{$f$-triangle} $f^{\triangle}(\Delta)$ that refines the $f$-vector of $\Delta$, with entries given by
$$f_{i,j}(\Delta) = |\{\sigma \in \Delta \mid d_{\Delta}(\sigma) = i-1, \; \dim(\sigma) = j-1 \}|.$$

Bj\"orner and Wachs also define a refinement of the $h$-vector called the $h$-triangle $h^{\triangle}(\Delta)$ which is a two-dimensional array with entries $h_{i,j}(\Delta)$ that is obtained from $f^{\triangle}(\Delta)$ by applying the $f$-vector to $h$-vector transformation on each row of $f^{\triangle}(\Delta)$. More precisely,
$$h_{i,j}(\Delta) = \sum_{k=0}^j(-1)^{j-k}\binom{i-k}{j-k}f_{i,k}.$$
The $f$- and $h$-triangles of a relative complex $(\Gamma,\Delta)$ are defined analogously.\footnote{Note that if $\dim(\Gamma) \neq \dim(\Delta)$ then the $f$-triangles of $\Gamma$ and $\Delta$ will have different dimensions.}

\begin{remark}If $\Gamma \supseteq \Delta$ with $\dim(\Gamma) = \dim(\Delta)$ and $d_{\Delta}(\sigma) = d_{\Gamma}(\sigma)$ for all $\sigma \in \Delta$, then $f^{\triangle}(\Gamma,\Delta) = f^{\triangle}(\Gamma) - f^{\triangle}(\Delta)$. Indeed, suppose that $\sigma \in \Gamma$ contributes to $f_{i,j}(\Gamma)$. Either $\sigma \in \Delta$, in which case by assumption it contributes to $f_{i,j}(\Delta)$, or $\sigma \in (\Gamma,\Delta)$, in which case it contributes to $f_{i,j}(\Gamma,\Delta)$. Since the $f$-triangle to $h$-triangle transformation is linear and $d_{\Delta}(\sigma) = d_{\Gamma}(\sigma)$, we also have $h^{\triangle}(\Gamma,\Delta) = h^{\triangle}(\Gamma) - h^{\triangle}(\Delta)$.
\end{remark}

It is natural to assume that the entries $h_{i,j}$ of the $h$-triangle of a partitionable nonpure complex have an analogous interpretation to the entries of the $h$-vector of a pure partitionable complex. This is \emph{false} in general. In \cite[Example~1]{Hachimori_sequential}, Hachimori gives an example of a partitionable nonpure complex whose $h$-triangle has a negative entry.

However, Hachimori introduces several strictly stronger variants of partitionability for nonpure complexes \cite{Hachimori_sequential}; among these is the existence of an \emph{$h$-compatible} partitioning of $\Delta$, i.e., a partitioning of the face poset of $\Delta$ where $h_{i,j}(\Delta)$ is the number of Boolean intervals in the partitioning whose bottom element is a face of size $j$ and whose top element is a facet of size $i$. In \cite[Theorem 2]{Hachimori_sequential}, Hachimori shows that $h$-compatibility is equivalent to a property he calls {layer-compatibility}: A partitioning 
$$P(\Delta) = \bigsqcup_{F \text{ facet of }\Delta}[\Psi(F),F]$$
of the face poset of $\Delta$ is \emph{layer-compatible} if the restriction
$$\bigsqcup_{\substack{F \text{ facet of } \Delta\\ \dim(F) \geq r}} [\Psi(F),F]$$
is a partitioning of the face poset of $\langle F \mid F \text{ facet of } \Delta, \; \dim(F) \geq r\rangle$ for all $0\leq r \leq \dim(\Delta)$.

\begin{remark}
While \cite[Theorem 2]{Hachimori_sequential} is stated in terms of absolute complexes, the same proof works for relative complexes as well. 
\end{remark}

We can now prove a nonpure analog of Theorem~\ref{main}. 

\begin{theorem}\label{Main_nonpure}
Let $\Delta$ be a nonpure complex. Then there is a complex $\Gamma \supseteq \Delta$ with $\dim(\Gamma) = \dim(\Delta)$ such that $\Gamma$ and $(\Gamma,\Delta)$ are layer-compatibly partitionable. 
\end{theorem}

\begin{proof}
Let $\Delta$ be a nonpure complex, and let $\Gamma$ be the complex obtained by attaching a $(d_{\Delta}(\sigma),k)$-partition extender to each $k$-face $\sigma$ of $\Delta$ for all $k$. Clearly, $d_{\Delta}(\sigma) = d_{\Gamma}(\sigma)$ for all $\sigma \in \Delta$, so we can write the $h^{\triangle}(\Delta)$ as the difference $h^{\triangle}(\Gamma) - h^{\triangle}(\Gamma,\Delta)$.

It is easy to check that the partitionings of $\Gamma$ and $(\Gamma,\Delta)$ we construct in Proposition~\ref{dk_part} are both layer-compatible.
\end{proof}

Since layer-compatibility implies $h$-compatibility, we now have a combinatorial interpretation of the $h$-triangle of any nonpure complex. We define an \emph{$(i,j)$-interval} of $\Delta$ to be a Boolean interval of $P(\Delta)$ whose bottom element has size $j$ and whose top element is a facet of size $i$.

\begin{corollary} For any nonpure complex $\Delta$, we have
\begin{align*}
h_{i,j}(\Delta) = &|\{(i,j)\text{-intervals in an } h \text{-compatible partitioning of }\Gamma\}|\\
& - |\{(i,j)\text{-intervals in an } h \text{-compatible partitioning of }(\Gamma,\Delta)\}|,
\end{align*}
where $\Gamma$ is the partition extender constructed in Theorem~\ref{Main_nonpure}.
\end{corollary}

\section{Cohen--Macaulay Extenders}\label{sec_CM}

Given the existence of partition extenders of pure simplicial complexes, it seems natural to ask if extenders exist for other well-studied combinatorial properties of simplicial complexes. A relative complex \((\Gamma,\Delta)\) is \defword{relative Cohen--Macaulay} if \(I_\Gamma / I_\Delta\) is a Cohen--Macaulay \(\Bbbk[\mathbf{x}]\)-module. Equivalently, a relative complex is relative Cohen--Macaulay if the relative homology \(\tilde{H}_i(\lk_\Gamma(\sigma),\lk_\Delta(\sigma))\) is trivial except possibly when \(|\sigma|+i=d\), where \(d\) is the dimension of \(\Gamma\) \cite[Theorem III.7.2]{Green_Stanley}.

\begin{definition}\label{def:cm_ext}
Let $\Delta$ be a pure $d$-dimensional simplicial complex. A  $d$-dimensional complex $\Gamma$ is a \defword{Cohen--Macaulay extender} for $\Delta$ if
\begin{itemize}
\item $\Delta \subseteq \Gamma$.
\item $\Gamma$ is Cohen--Macaulay.
\item The relative complex $(\Gamma, \Delta)$ is relative Cohen--Macaulay.
\end{itemize}
\end{definition}

Unlike the case for partition extenders, there is a large class of pure complexes for which Cohen--Macaulay extenders do not exist. The \emph{depth} of a simplicial complex \(\Delta\) is defined as \(\depth \Bbbk[\Delta]\), the depth of its Stanley--Reisner ring. By applying Hochster's formula \cite{Hochster}, it can be shown that \(\depth \Bbbk[\Delta]\) is the largest integer \(h\) such that \(\tilde{H}_i(\lk_\Delta(\sigma))\) is trivial whenever $|\sigma| + i + {1} < h $ for all $ -1 < i < d$ and $\sigma \in \Delta$. We recall that for a \(d\)-dimensional simplicial complex \(\Delta\), \(\dim \Bbbk[\Delta]= d+1\).

\begin{proposition}\label{Prop::NoDice}
If $\Delta$ is a simplicial complex such that $\depth \Bbbk[\Delta] < \dim \Bbbk[\Delta] - 1$, then $\Delta$ does not have a Cohen--Macaulay extender.
\end{proposition}

\begin{proof}
Let $\Delta$ be a $d$-dimensional complex with $\depth \Bbbk[\Delta] < \dim \Bbbk[\Delta] - 1$. By definition, there is a face $\sigma \in \Delta$ and an index \(i\) such that \(\tilde{H}_i(\lk_\Delta(\sigma))\) is nontrivial where \(|\sigma|+i+{1} = \depth \Bbbk[\Delta] \le d-1\); equivalently, $i+1 \le d - \abs{\sigma} -1$. 

Suppose \(\Gamma\) is a \(d\)-dimensional complex such that \(\Gamma\) is Cohen--Macaulay and \(\Delta \subseteq \Gamma\). We can write the long exact sequence of relative homology for the pair \((\lk_\Gamma(\sigma),\lk_\Delta(\sigma))\).

\[\begin{tikzpicture}[descr/.style={fill=white,inner sep=1.5pt},scale=1]
        \matrix (m) [
            matrix of math nodes,
            row sep=1em,
            column sep=1.1em,
            text height=1.5ex, text depth=0.25ex
        ]
        { 0 &  \tilde{H}_{d-|\sigma|}(\lk_\Delta(\sigma))  & \tilde{H}_{d-|\sigma|}(\lk_\Gamma(\sigma)) & \tilde{H}_{d-|\sigma|}((\lk_\Gamma(\sigma),\lk_\Delta(\sigma))) \\
          \mbox{}  &   \tilde{H}_{d-|\sigma|-1}(\lk_\Delta(\sigma))  & \tilde{H}_{d-|\sigma|-1}(\lk_\Gamma(\sigma)) & \tilde{H}_{d-|\sigma|-1}((\lk_\Gamma(\sigma),\lk_\Delta(\sigma)))  \\
          \mbox{}  &  \tilde{H}_{d-|\sigma|-2}(\lk_\Delta(\sigma))  & \tilde{H}_{d-|\sigma|-2}(\lk_\Gamma(\sigma)) & \tilde{H}_{d-|\sigma|-2}((\lk_\Gamma(\sigma),\lk_\Delta(\sigma))) \\
         \mbox{}  & \tilde{H}_{d-|\sigma|-3}(\lk_\Delta(\sigma))  & \tilde{H}_{d-|\sigma|-3}(\lk_\Gamma(\sigma))  & \mbox{} &  \\
        };

        \path[overlay,->, font=\scriptsize,>=latex]
        (m-3-4) edge[out=355,in=175] (m-4-2)
        (m-4-2) edge (m-4-3)
        (m-4-3) edge[dashed] (m-4-4);
        \path[overlay,->, font=\scriptsize,>=latex]
        (m-3-2) edge (m-3-3)
        (m-3-3) edge (m-3-4);
        \path[overlay,->, font=\scriptsize,>=latex]
        (m-1-4) edge[out=355,in=175] (m-2-2)
        (m-2-2) edge (m-2-3)
        (m-2-3) edge (m-2-4)
        (m-2-4) edge[out=355,in=175] (m-3-2);
        \path[overlay,->, font=\scriptsize,>=latex]
        (m-1-3) edge (m-1-4)
        (m-1-2) edge (m-1-3)
        (m-1-1) edge (m-1-2);
\end{tikzpicture}\]

Since $\Gamma$ is Cohen--Macaulay, we know that \(\tilde{H}_j(\lk_\Gamma(\sigma))\) is trivial whenever \(j < d - |\sigma|\). This observation lets us break up the long exact sequence into the following exact sequences for each \(\ell \geq 1\):
\[\begin{tikzpicture}[descr/.style={fill=white,inner sep=1.5pt}]
        \matrix (m) [
            matrix of math nodes,
            row sep=1em,
            column sep=1.8em,
            text height=1.5ex, text depth=0.25ex
        ]
        { 0 & \tilde{H}_{d-|\sigma|-\ell}((\lk_\Gamma(\sigma),\lk_\Delta(\sigma))) & \tilde{H}_{d-|\sigma|-\ell-1}(\lk_\Delta(\sigma)) & 0 \\
        };

        \path[overlay,->, font=\scriptsize,>=latex]
        (m-1-1) edge (m-1-2)
        (m-1-2) edge (m-1-3)
        (m-1-3) edge (m-1-4);
\end{tikzpicture}\]
Each of these middle maps is an isomorphism. Since \(\tilde{H}_i(\lk_\Delta(\sigma))\) is nontrivial, \(\tilde{H}_{i+1}((\lk_\Gamma(\sigma),\lk_\Delta(\sigma)))\) is also nontrivial.
Since \(i+1 \le d - |\sigma|-1\), 
the relative complex \((\Gamma, \Delta)\) 
is not
relative Cohen--Macaulay. Therefore there is no Cohen--Macaulay extender for \(\Delta\).
\end{proof}

\begin{theorem}\label{thm::CM}
Let $\Delta$ be a simplicial complex. Then $\Delta$ has a Cohen--Macaulay extender if and only if $\depth\Bbbk [\Delta] \ge \dim\Bbbk [\Delta] - 1$.
\end{theorem}

\begin{proof}
The case that $\depth\Bbbk [\Delta] < \dim\Bbbk [\Delta] - 1$ is covered by Proposition~\ref{Prop::NoDice}, so we assume that $\depth\Bbbk [\Delta] \ge \dim\Bbbk [\Delta] - 1$.

Let \(\Delta\) be a \(d\)-dimensional simplicial complex with depth at least $d$, and let \(\Gamma\) be a Cohen--Macaulay \(d\)-dimensional complex that contains \(\Delta\). We begin by writing a short exact sequence of modules over \(\Bbbk [x_1,\ldots,x_n]\) with \(I_\Delta\) and \(I_\Gamma\) as the Stanley--Reisner ideals associated to \(\Delta\) and \(\Gamma\).
\[0 \to I_\Delta / I_\Gamma \to \Bbbk[\Gamma] \to \Bbbk[\Delta] \to 0\]
By the assumptions on \(\Delta\) and \(\Gamma\), we can see that \(\depth\Bbbk[\Gamma]=\dim\Bbbk[\Gamma]\) and \(\depth\Bbbk[\Delta] \geq \dim\Bbbk[\Delta] - 1 = \dim\Bbbk[\Gamma] - 1 \). By the Depth Lemma \cite[Proposition 1.2.9]{Bruns_Herzog}, we get that \(\depth(I_\Delta / I_\Gamma)=\dim\Bbbk[\Gamma] - 1 \). This is equivalent to saying that \((\Gamma,\Delta)\) is relative Cohen--Macaulay. Therefore \(\Gamma\) is a Cohen--Macaulay extender of \(\Delta\).
\end{proof}

Theorem~\ref{thm::CM} shows that if $\depth \Bbbk[\Delta] \geq \dim \Bbbk[\Delta] - 1$, then \emph{any} Cohen--Macaulay complex $\Gamma$ of the same dimension that contains $\Delta$ is a Cohen--Macaulay extender for $\Delta$. If $\Delta$ is a $d$-dimensional complex on $n+1$ vertices, then perhaps the most natural Cohen--Macaulay extender to consider is the $d$-skeleton of the $n$-simplex $\Delta^{(d)}_n$, which is 
$$
\Delta^{(d)}_n = \{ \sigma \subseteq [n+1] ~:~ |\sigma| \le d+1 \}.
$$

In particular, we note that if a Cohen--Macaulay extender exists for a complex, then we can construct one without introducing new vertices.

Codenotti, Katth\"{a}n, and Sanyal recently classified the $h$-vectors of relative Cohen--Macaulay complexes. In \cite[Theorem 5.7]{CKS}, it is shown that $(h_0,\dots,h_{d+1})$ is the $h$-vector of a proper Cohen--Macaulay relative complex if and only if $h_0=0$ and $h_i \ge 0$ for all $i$, answering a question of Bj\"{o}rner in \cite{St87}. (Here ``proper'' means that the subcomplex in question is not the void complex.) They find more a restrictive characterization in \cite[Theorem 1.3]{CKS} for Cohen--Macaulay relative complexes on ground set $[n]$. Theorem~\ref{thm::CM} is a result in the same vein, with the further constraint that the total complex be Cohen--Macaulay.

\section{Shelling extenders and Simon's conjecture}\label{sec_shell}

A relative complex $(\Gamma,\Delta)$ is \defword{shellable} if its facets can be ordered $F_1,\dots,F_k$ such that $\ideal{F_{i+1}}\setminus\ideal{F_1,\dots,F_i,\Delta}$ has a unique minimal face for all $i\in [k-1]$. Such an ordering of the facets is a \defword{shelling order}. If a pure relative complex is shellable, then it is relative Cohen--Macaulay \cite[Page 118]{Green_Stanley}. Therefore, in our search for a similar notion of an extender for shellability,  we limit our search  to complexes $\Delta$ such that $\depth\Bbbk[\Delta] \geq \dim \Bbbk[\Delta]-1.$

\begin{definition}
Let $\Delta$ be a pure $d$-dimensional simplicial complex. A  $d$-dimensional complex $\Gamma$ is a \defword{shelling extender} for $\Delta$ if
\begin{itemize}
\item $\Delta \subseteq \Gamma$.
\item $\Gamma$ is shellable.
\item The relative complex $(\Gamma, \Delta)$ is shellable.
\end{itemize}
\end{definition}

\begin{conjecture}\label{conj:shell}
If $\Delta$ is a simplicial complex such that $\depth \Bbbk[\Delta] \ge \dim \Bbbk[\Delta] - 1$ for all fields $\Bbbk$, then $\Delta$ has a shelling extender.
\end{conjecture}

Such shellable extenders may have application to a conjecture of Simon.
We first recall that a pure complex $\Delta$ is \emph{extendably shellable} if every partial shelling order $F_1,\dots, F_j$ can be extended to a shelling order $F_1,\dots, F_j,F_{j+1},\dots,F_k$ of $\Delta$.

\begin{conjecture}{\rm \cite[Conjecture 4.2.1]{Simon_Cleanness}}\label{conj:Simon}
If $\Delta$ is the $d$-skeleton of an $n$-simplex, then $\Delta$ is extendably shellable.
\end{conjecture}

Some partial results about extendable shellability are known. Simon's conjecture is known to be true in certain cases. For $d\leq 1$ and $d\geq n-1$, the conjecture is clearly true. The case $d = n-2$ was proved by Bigdeli, Yazdan Pour, and Zaare-Nahandi in \cite{Extendable_Shellability_Clutters_3} and by Dochtermann in \cite{Extendable_Shellability_Clutters} (and was strengthened by Culbertson, Dochtermann, Guralnik and Stiller in \cite{Extendable_Shellability_Clutters_2}). 

The case $d=2$ was shown by Bj\"{o}rner and Eriksson in \cite{Extendable_Shellability_Matroids} as a consequence of the fact that matroid complexes of rank $\leq 3$ are extendably shellable, since the $d$-skeleton of the $n$-simplex is the independence complex of the uniform matroid of rank $d+1$ over $n+1$ elements. On the other hand, in \cite[Theorem 2.3.1]{Hall_Counterexamples_Discrete_Geom} Hall shows that the boundary of the $d$-crosspolytope is not extendably shellable for $d\geq 12$. In \cite{Bruno_Counterexample}, Benedetti and Bolognini found a counterexample to a strengthening of Simon's conjecture that had been posed by Bigdeli and Faridi \cite{Chordality_d-collapsibility}, Dochtermann \cite{Extendable_Shellability_Clutters}, and Nikseresht \cite{Chordality_of_clutters}.

We note the connection between Conjecture~\ref{conj:shell} and Simon's conjecture.

\begin{question}\label{question:unlikely}
If a shelling extender exists for $\Delta$, then is it possible to create a shelling extender $\Gamma$ without introducing any new vertices?
\end{question}

\begin{remark}\label{remark:hope}
If Question~\ref{question:unlikely} has a positive answer, then this would prove Conjecture~\ref{conj:Simon}.
\end{remark}

Theorem~\ref{thm::CM} shows that the $d$-skeleton of the $n$-simplex is a Cohen--Macaulay extender for $\Delta$ whenever such an extender exists. Thus it is reasonable to ask whether this construction is possible in the case of shelling extenders. We note that the $h$-vector characterizations of shellable relative complexes is the same as in the Cohen--Macaulay case \cite{CKS}, so there is no direct numerical obstruction to this construction.

\section{Questions and Future Directions}\label{sec_closing}

One may ask how close a given complex $\Delta$ is to being partitionable by considering the ``smallest'' possible partition extender $\Gamma$. Our construction produces partition extenders that are quite large, but it is often possible to find smaller extenders by hand. The bow-tie pictured below is a standard example of a non-partitionable complex, with a negative entry in the \(h\)-vector.
\begin{example}\label{ex:MinimalBowtie}
Below, the dark complex is the bow-tie with \(f\)-vector equal to \((1,5,6,2)\) and \(h\)-vector equal to \((1,2,-1,0)\). The entire complex pictured has \(f\)-vector \((1,5,7,3)\) and \(h\)-vector \((1,2,0,0)\). The lighter shaded relative complex has \(f\)-vector \((0,0,1,1)\) and \(h\)-vector \((0,0,1,0)\). Both the larger complex and relative complex are partitionable,  and the $h$-vector of the bow-tie is given by the difference of the two other $h$-vectors.

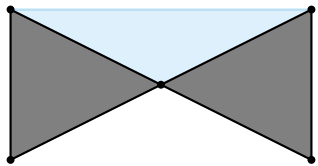
\begin{center}
\begin{tikzpicture}
\node[coordinate] at (-2,-1) (e1){};
\node[coordinate] at (-2,1) (e2){};
\node[coordinate] at (0,0) (e3){};
\node[coordinate] at (2,1) (e4){};
\node[coordinate] at (2,-1) (e5){};
\draw[fill=black!50] (e1) -- (e2) -- (e3) -- (e1);
\draw[fill=black!50] (e3) -- (e4) -- (e5) -- (e3);
\draw[fill=skyblue!20] (e2) -- (e3) -- (e4) -- (e2);

\draw[color=skyblue!40, thick] (e2) -- (e4);
\draw[thick] (e1) -- (e2) -- (e3) -- (e4) -- (e5) -- (e3) -- (e1);
\node[circle,color=black,fill=black,inner sep=1pt,minimum size=3pt] at (e1){};
\node[circle,color=black,fill=black,inner sep=1pt,minimum size=3pt] at (e2){};
\node[circle,color=black,fill=black,inner sep=1pt,minimum size=3pt] at (e3){};
\node[circle,color=black,fill=black,inner sep=1pt,minimum size=3pt] at (e4){};
\node[circle,color=black,fill=black,inner sep=1pt,minimum size=3pt] at (e5){};
\end{tikzpicture}
\end{center}

\end{example}

The above example of a partition extender is far smaller than those constructed in the proof of Theorem~\ref{main}. This observation leads naturally to the following questions:
\begin{question} Is it possible to construct a minimal partition extender with respect to the number of faces added? With respect to the size of the $h$-vector of the relative complex?
With respect to some other measure of size? \end{question}

\begin{question} Assuming that a minimal partition extender exists, is it unique? \end{question}

If, for example, $\Delta$ is a complete graph on four vertices together with two additional disjoint edges, then $h(\Delta) = (1,6,1)$ but $\Delta$ is not partitionable. This means that the number and sizes of the negative entries of the $h$-vector of a complex does not capture how many faces need to be added to create a partition extender, since there are non-partitionable complexes whose $h$-vectors are all positive. In fact, a result of Duval, Goeckner, Klivans, and Martin \cite{DG16} shows that that there are even Cohen--Macaulay complexes (which have much stronger conditions on their $h$-vectors than positivity) that are non-partitionable.

\begin{example}
Here we explicitly realize our construction on a pair of edges in black, with the partition extender drawn in a lighter shade. Our construction adds $8$ vertices and $13$ edges, but a minimal partition extender can be created by introducing a single edge to connect the two edges in black.

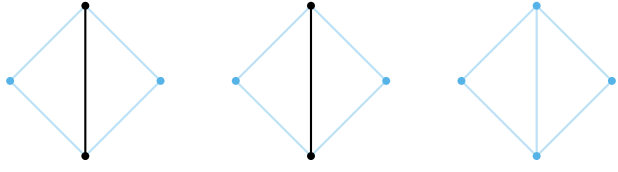
\begin{center}
\begin{tikzpicture}
\node[coordinate] at (-4,0) (e1){};
\node[coordinate] at (-3,-1) (e2){};
\node[coordinate] at (-3,1) (e3){};
\node[coordinate] at (-2,0) (e4){};
\node[coordinate] at (-1,0) (e5){};
\node[coordinate] at (0,1) (e6){};
\node[coordinate] at (0,-1) (e7){};
\node[coordinate] at (1,0) (e8){};
\node[coordinate] at (2,0) (e9){};
\node[coordinate] at (3,1) (e10){};
\node[coordinate] at (3,-1) (e11){};
\node[coordinate] at (4,0) (e12){};
\draw[thick] (e3) -- (e2);
\draw[thick] (e6) -- (e7);
\draw[thick, color=skyblue!40] (e1) -- (e2) -- (e4) -- (e3) -- (e1);
\draw[thick, color=skyblue!40] (e5) -- (e6) -- (e8) -- (e7) -- (e5);
\draw[thick, color=skyblue!40] (e9) -- (e10) -- (e12) -- (e11) -- (e10);
\draw[thick, color=skyblue!40] (e9) -- (e11);

\node[circle,color=black,fill=black,inner sep=1pt,minimum size=3pt] at (e2){};
\node[circle,color=black,fill=black,inner sep=1pt,minimum size=3pt] at (e3){};
\node[circle,color=black,fill=black,inner sep=1pt,minimum size=3pt] at (e6){};
\node[circle,color=black,fill=black,inner sep=1pt,minimum size=3pt] at (e7){};
\node[circle,color=skyblue,fill=skyblue,inner sep=1pt,minimum size=3pt] at (e1){};
\node[circle,color=skyblue,fill=skyblue,inner sep=1pt,minimum size=3pt] at (e4){};
\node[circle,color=skyblue,fill=skyblue,inner sep=1pt,minimum size=3pt] at (e5){};
\node[circle,color=skyblue,fill=skyblue,inner sep=1pt,minimum size=3pt] at (e8){};
\node[circle,color=skyblue,fill=skyblue,inner sep=1pt,minimum size=3pt] at (e9){};
\node[circle,color=skyblue,fill=skyblue,inner sep=1pt,minimum size=3pt] at (e10){};
\node[circle,color=skyblue,fill=skyblue,inner sep=1pt,minimum size=3pt] at (e11){};
\node[circle,color=skyblue,fill=skyblue,inner sep=1pt,minimum size=3pt] at (e12){};
\end{tikzpicture}
\end{center}

\end{example}

Given a complex $\Delta$, we might ask for an upper bound on how many faces must be added to create a partition extender $\Gamma$ via our construction. If $g(k)$ is the number of faces in a \((d,d-k)\)-partition extender, then $g(k)$ is defined by the recurrence relation
\[
g(k)=k(2^{d+1}-2^k)+ \sum_{j=0}^{k-1} \binom{k}{j} g(j).
\]
Since $g$ is an increasing function, if we ignore the term \(-2^{k}\), we obtain a simple one-term recurrence relation bound
\[
g(k) \leq k2^{d+1} + 2^{k}g(k-1).
\]
As long as \(g(k-1)>2^{d+1}\),
\[
g(k) \leq 2(2^k)g(k-1).
\]
The starting term is \(g(0)=0\), and \(g(1) \leq 2^{d+1}\). Therefore, an upper bound for \(g(k)\) is
\[
g(k) \leq 2^{2^{k}-1+d}.
\]
Thus, given a complex $\Delta$ with $f(\Delta)=(f_{-1},f_0,\ldots,f_{d})$, our construction will add
$$
\sum_{-1\le k \le d} f_{k}\cdot g(d-k) \le \sum_{-1\le k \le d} f_k \cdot 2^{2^{d-k}-1+d}
$$
total faces. This bound is not exact, but we expect it to be of the correct order of magnitude. As seen in Example~\ref{ex:MinimalBowtie}, the number of faces added in a minimal partition extender can be much lower.

In Section~\ref{sec_nonpure_partition} we constructed nonpure partition extenders. Along the same lines, given some condition on the depths of the pure skeletons of a nonpure complex $\Delta$, we expect that it should be possible to construct a \emph{sequentially Cohen--Macaulay extender} $\Gamma$, that is, a $\Gamma \supseteq \Delta$ such that $d_{\Delta}(\sigma) = d_{\Gamma}(\sigma)$ for all $\sigma \in \Delta$, and $\Gamma$ and $(\Gamma,\Delta)$ are both sequentially Cohen--Macaulay.

\section*{Acknowledgments}

We thank the referees for their helpful suggestions, especially for those regarding the subtleties of nonpure partitionability. We also thank Margaret Bayer for her careful reading of earlier drafts. B.~Goeckner was partially supported by an AMS-Simons travel grant.

\bibliographystyle{amsplain}
\bibliography{Extenders}
\end{document}